\documentclass{amsart}
\usepackage{omri_commands}
\usepackage{nicematrix}
\usepackage{pgf,tikz}
\usepackage{mathrsfs}
\usetikzlibrary{arrows}
\definecolor{uuuuuu}{rgb}{0.26666666666666666,0.26666666666666666,0.26666666666666666}
\definecolor{black}{rgb}{0.6,0.2,0.}
\definecolor{xdxdff}{rgb}{0.49019607843137253,0.49019607843137253,1.}

\title{Orbits in Teichm{\"u}ller dynamics admits a critical exponent gap}
\author{Omri Nisan Solan}
\thanks{This research was supported by ERC 2020 grant HomDyn (grant no.~833423).}
\date{\today}
\usepackage{todonotes}
\begin{document}
\begin{abstract}
  McMullen '03 constructs a collection of orbits $\SL_2(\RR).x$ in $\cH(1,1)$ with infinitely generated stabilizers $\stab_{\SL_2(\RR)}(x)$.
  We prove a gap in the set of critical exponents of stabilizers of $\SL_2(\RR)$-orbits in $\cH_g$: for every $x\in \cH_g$, either $\stab_{\SL_2(\RR)}(x)$ is a lattice, or we have a uniform bound on the critical exponent $\delta(\stab_{\SL_2(\RR)}(x)) \le 1-\varepsilon_g$.
\end{abstract}

\maketitle
\section{Introduction}
Let $\cH_g$ denote the space of flat surfaces with genus $G$ and volume $1$. This space has an $\SL_2(\RR)$ action.
This space and action are related to multiple areas in mathematics. On the one hand, every element $x\in \cH_g$ represents a complex curve,
relating the flat surface dynamics to algebraic geometry and number theoretic properties of subvarieties of \teich\ spaces.
On the other hand, the dynamics are related to interval exchange transformations and billiard dynamics.

Because of these relations, the dynamics of $\SL_2(\RR)\acts \cH_g$ where studied from different perspectives. A long line of research aims to study invariant subsets of $\cH_g$.
In Eskin Mirzakhani \cite{eskin2018invariant}, the authors prove a general measure classification result, showing that every $\begin{pmatrix}
  *&*\\ 0&*
\end{pmatrix}$-invariant probability measure is $\SL_2(\RR)$ invariant and is the Lebesgue measure on a certain kind of manifold called \emph{affine manifolds}. Using the measure classification result, Eskin Mirzakhani and Mohammadi \cite{eskin2015isolation}, showed that every invariant irreducible closed subset is an affine manifold.
Classification of invariant subsets was established for $g=2$ by McMullen \cite{mcmullen2003genus2dynamics}, but not for higher genuses.
There has been extensive study on the possible affine submanifolds of $\cH_g$, particularly the potential closed orbits.
However, in \cite{mcmullen2003teichmuller}, McMullen describes a collection of non-closed orbits with infinite complexity. See Calta \cite{calta2004veech} for similar results in this direction.
\begin{figure}
  \begin{center}
    \begin{tikzpicture}[line cap=round,line join=round,>=triangle 45,x=3.0cm,y=3.0cm]
      \clip(-1.3094792742422154,-0.05) rectangle (2.7029251339713363,1.9002304146042626);
      \fill[color=black,fill=black,fill opacity=0.1] (0.,0.) -- (0.,1.) -- (-1.,1.) -- (-1.,0.) -- cycle;
      \fill[color=black,fill=black,fill opacity=0.1] (0.,1.676073743754346) -- (0.,0.) -- (1.6760737437543458,0.) -- (1.676073743754346,1.6760737437543458) -- cycle;
      \fill[color=black,fill=black,fill opacity=0.1] (1.676073743754346,1.6760737437543458) -- (1.6760737437543458,1.) -- (2.352147487508691,1.) -- (2.3521474875086916,1.6760737437543456) -- cycle;
      \draw [color=black] (0.,0.)-- (0.,1.);
      \draw [color=black] (0.,1.)-- (-1.,1.);
      \draw [color=black] (-1.,1.)-- (-1.,0.);
      \draw [color=black] (-1.,0.)-- (0.,0.);
      \draw [color=black] (0.,1.676073743754346)-- (0.,0.);
      \draw [color=black] (0.,0.)-- (1.6760737437543458,0.);
      \draw [color=black] (1.6760737437543458,0.)-- (1.676073743754346,1.6760737437543458);
      \draw [color=black] (1.676073743754346,1.6760737437543458)-- (0.,1.676073743754346);
      \draw [color=black] (1.676073743754346,1.6760737437543458)-- (1.6760737437543458,1.);
      \draw [color=black] (1.6760737437543458,1.)-- (2.352147487508691,1.);
      \draw [color=black] (2.352147487508691,1.)-- (2.3521474875086916,1.6760737437543456);
      \draw [color=black] (2.3521474875086916,1.6760737437543456)-- (1.676073743754346,1.6760737437543458);
      \begin{scriptsize}
      \draw[color=black] (-0.5025152591825067,0.5300727640341384) node {\scalebox{3}{$1$}};
      \draw[color=black] (0.8480286826882558,0.8326842696815278) node {\scalebox{3}{$1+a$}};
      \draw[color=black] (2.041662954964075,1.3538485294075873) node {\scalebox{3}{$a$}};
      \end{scriptsize}
    \end{tikzpicture}
    \caption{The surface $S(a)$.}
    \label{fig: Sa}
  \end{center}
\end{figure}
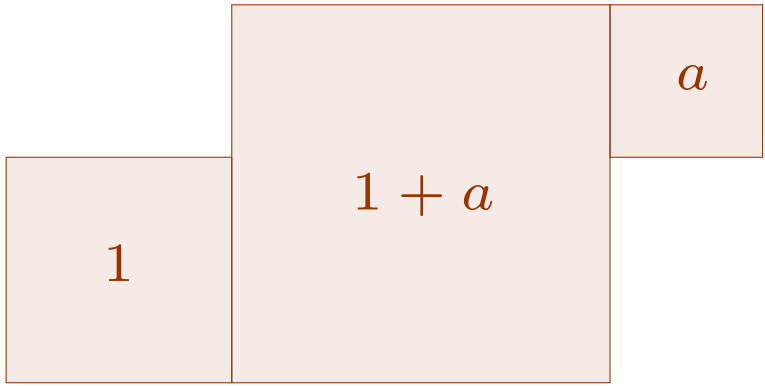
\begin{theorem}[McMullen \cite{mcmullen2003teichmuller}]\label{thm: McMullen}
  Let $b$ be a rational number such that $a = b-1 + \sqrt{b^2 - b + 1}$ is irrational, and let $(S(a),\bd z)\in \cH(1,1)$ be the flat surface defined by the polygon in Figure \ref{fig: Sa}, where we glue opposite edges. 
  Then $\stab_{\SL_2(\RR)}((S(a),\bd z))$ is infinitely generated. 
\end{theorem}
Since then these infinite complexity orbits remained mysterious and there were no theoretical results about their nature. 
The main goal of the paper is to prove the following.
\begin{theorem}\label{thm: main}
  Let $g>0$. There exists an $\varepsilon_g$ such that for each $x\in \cH_g$, 
  either $\stab_{\SL_2(\RR)}(x)$ is a lattice in $\SL_2(\RR)$, or
  we have a bound on the critical exponent (see Definition \ref{def: critical exponent})
  \[\delta(\stab_{\SL_2(\RR)}(x)) \le 1-\varepsilon_g. \]
\end{theorem}
Comparing \teich\ dynamics to homogeneous dynamics, Theorem \ref{thm: main} has an analogous result \cite{solan2024critical}:
\begin{theorem}[Critical exponent gap in homogeneous dynamics]\label{thm: old gap}
  Let $\Gamma<\SL_2(\CC)$ be a geometrically finite Kleinian group.
  Then there is an $\varepsilon_\Gamma>0$ such that for every non-periodic orbit $\SL_2(\RR).x$ in $\SL_2(\CC)/\Gamma$ we have a bound on the critical exponent
  \[\delta(\stab_{\SL_2(\RR)}(x)) \le 1-\varepsilon_\Gamma.\]
\end{theorem}
In \cite{solan2024critical} we announce that Theorem \ref{thm: old gap} serves as a replacement to the arithmeticity assumption in Lindenstrauss, Mohammadi, and Wang \cite{lindenstrauss2022effective} and prove a polynomial equidistribution result of $\left\{\begin{pmatrix}
  1&s\\0&1
\end{pmatrix}:s\in \RR\right\}$-orbits in nonarithmetic lattice quotients $\SL_2(\CC)/\Gamma$.
Hence Theorem \ref{thm: main} seems like a good starting point to show a polynomial equidistribution result in \teich\ dynamics. See \cite{sanchez2023effective, rached2024separation} for other results in this direction.

A drawback of Theorem \ref{thm: main} is the lack of control over $\varepsilon_g$. This is a consequence of the ergodic nature of the proof. Thus a natural question is the following
\begin{ques}
  Give an estimate of $\varepsilon_g$. 
  Find the minimal $a_g\ge 0$ so that for every $\zeta>0$ there are at most finitely many orbits $\SL_2(\RR).x$ in $\cH_g$ with $\delta(\stab_{\SL_2(\RR)}(x)) \ge a_g + \zeta$ such that $\SL_2(\RR).(x)$ is not periodic.
\end{ques}

\subsection{On the proof}
The overall structure of the proof follows a method previously employed by Eskin, Filip, and Wright \cite{eskin2018algebraichull}, as well as by Bader, Fisher, Miller, and Stover \cite{bader2021arithmeticity}.
While working in different settings, \cite{eskin2018algebraichull} in \teich\ dynamics and \cite{bader2021arithmeticity} in nonarithmetic homogeneous dynamics, both papers prove a finiteness result on certain collections submanifolds. To do so, both works follow the same general scheme, which we will call \emph{the impossible section method}.
We will first describe its simplest version.
\begin{proposition}[Impossible section method for $\SL_2(\RR)$-periodic orbits]\label{prop: impossible section method}
  Assume that $\SL_2(\RR) \acts M$ is an action of $\SL_2(\RR)$ on a locally compact second countable topological space $M$, with an invariant probability measure $\lambda$.
  Assume that $(\SL_2(\RR).x_i)_{i=1}^\infty$ is an infinite sequence of periodic orbits.
  Denote by $\mu_i$ the invariant probability measure on $\SL_2(\RR).x_i$.
  Assume that the following three parts hold:
  \begin{enumerate}[label=ISM-\alph*), ref=(ISM-\alph*)]
    \item \label{part: equidistribution}We have a weak-$*$ convergence $\mu_i \xrightarrow{i\to \infty} \lambda$.
    \item\label{part: lifts to projective} There are a projective bundle $\pi: P\to M$ with a compatible $\SL_2(\RR)$ action, and $\SL_2(\RR)$-equivariant sections $\varphi_i:\SL_2(\RR).x_i\to P$ such that $\pi\circ \varphi_i = {\rm Id}_{\SL_2(\RR).x_i}$.
    \item \label{part: no section}There are no $\SL_2(\RR)$-invariant probability measures on $P$ that project to $\lambda$ on $M$
  \end{enumerate}
  Then we get a contradiction.
\end{proposition}
\begin{remark}
  Part \ref{part: lifts to projective} can be interpreted as finding
  $\varphi_i(x_i) = \tilde x_i \in \pi^{-1}(x_i)$ satisfying that $\stab_{\SL_2(\RR)}(\tilde x_i) = \stab_{\SL_2(\RR)}(x_i)$. This will not be the way we think on this part here.
\end{remark}
\begin{proof}
  Since the bundle $P$ is projective, we can restrict to a subsequence such that $(\varphi_i)_*(\mu_i)$ has a weak-$*$ limit measure $\tilde \mu$ on $P$.
  Then as a limit of $\SL_2(\RR)$-invariant measure, $\tilde \mu$ is invariant and projects to
  \[\pi_* \tilde \mu = \pi_* \lim_{i}(\psi_i)_*\mu_i = \lim_{i}\pi_*(\psi_i)_*\mu_i =  \lim_{i}\mu_i = \lambda,\]
  where the commutativity of the limit and $\pi_*$ holds since $\pi$ is proper.
  This contradicts part \ref{part: no section}.
\end{proof}
There are variations of this method:
In \cite{eskin2018algebraichull}, the authors use a variation of the method with $M$ being an affine manifold with certain properties and show that there are finitely many maximal affine submanifolds of $M$. Thus, instead of orbits $\SL_2(\RR).x_i$, the authors assume infinitely many general affine submanifolds $\cM_i\subsetneq M$ in the impossible section method, with the Lebesgue measures $\mu_i = \lambda_{\cM_i}$ on them.
To prove Part \ref{part: equidistribution}, \cite{eskin2018algebraichull} uses Eskin, Mirzakhani and Mohammadi \cite{eskin2015isolation}.
The construction of the projective bundle $P$ in Part \ref{part: lifts to projective} uses the Kontsevich-Zorich cocycle and the sections $\varphi_i$ use number-theoretic properties of the manifolds.
The proof of Part \ref{part: no section} involves calculating the algebraic hull of the cocycle that defines $P$.

Although Bader, Fisher, Miller, and Stover \cite{bader2021arithmeticity} proved a more general result, we will focus on the case $\SL_2(\CC) / \Gamma$ where $\Gamma<\SL_2(\CC)$ is a nonarithmetic lattice.
They show that there are only finitely many periodic $\SL_2(\RR)$ orbits in $M=\SL_2(\CC)/\Gamma$.
To do so they use impossible section method described above, implementing it with various techniques. Here, we focus on how they proved Part \ref{part: equidistribution} in their setting, utilizing the work of Mozes and Shah \cite{mozes1995space}.

The goal of Solan \cite{solan2024critical} is to prove Theorem \ref{thm: old gap}.
For this, we use a different variation of the impossible section method.
For every $t,s\in \RR$ denote by
\begin{align}\label{eq: matrices in SL2}
  \ta(t) = \begin{pmatrix}
    e^{t/2}&0\\0&e^{-t/2}
  \end{pmatrix},\qquad
  \tu(s) = \begin{pmatrix}
    1&s\\0&1
  \end{pmatrix}.
\end{align}
\begin{proposition}[Impossible section method with $\SL_2(\RR)$-high critical exponent orbits]\label{prop: impossible section method1}
  Assume that $\SL_2(\RR) \acts M$ is an action of $\SL_2(\RR)$ on a locally compact second countable topological space $M$, with an invariant probability measure $\lambda$.
  Assume that $(\SL_2(\RR).x_i)_{i=1}^\infty$ is an infinite sequence of orbits, with $\delta(\stab_{\SL_2(\RR)}(x_i))\to 1$.
  Then there are $\ta$-invariant probability measures $\mu_i \in \Prob(\SL_2(\RR).x_i)$ with entropies $h_{\mu_i}(\ta(1)) \to 1$.
  If the following three-part method holds, then we get a contradiction.
  \begin{enumerate}[label=\emph{ISM'-\alph*)}, ref=(ISM'-\alph*)]
    \item \label{part: equidistribution1}We have a weak-$*$ convergence $\mu_i \xrightarrow{i\to \infty} \lambda$.
    \item\label{part: lifts to projective1} There is a projective bundle $\pi: P\to M$ with a compatible $\SL_2(\RR)$ action of $G$, and an $\SL_2(\RR)$-equivariant sections $\varphi_i:\SL_2(\RR).x_i\to P$ such that $\pi\circ \varphi_i = {\rm Id}_{\SL_2(\RR).x_i}$.
    \item \label{part: no section1}There are no $\SL_2(\RR)$-invariant probability measures on $P$ that project to $\lambda$ on $M$
  \end{enumerate}
\end{proposition}
In Appendix \ref{appendix: proof impossible section strat} we will deduce Proposition \ref{prop: impossible section method1} from results in \cite{solan2024critical}.
\begin{discussion}
  Heuristically, the variation of Proposition \ref{prop: impossible section method1} from Proposition \ref{prop: impossible section method} is the following observation: we do not need the measures $\mu_i$ themselves to be $\SL_2(\RR)$ invariant, only \emph{sufficiently invariant} to ensure that the partial limits $\lim \mu_i$ and $\lim(\varphi_i)_*\mu_i$ are invariant.
  While the invariance assumption on $\lim \mu_i$ is not crucial to the proof of the proposition, it is crucial to the proof of Part \ref{part: equidistribution1} in the implementations.
\end{discussion}
Part \ref{part: lifts to projective1} in \cite{solan2024critical} is very similar to the same part in \cite{bader2021arithmeticity}, and part \ref{part: no section1} is identical.
However, for part \ref{part: equidistribution1} Solan \cite{solan2024critical} had to adapt the Mozes-Shah Linearization argument.

The relation between \cite{solan2024critical} and \cite{bader2021arithmeticity} is similar to the relation between this work and \cite{eskin2018algebraichull}. Here we also use this variation of the impossible section method.
Part \ref{part: lifts to projective1} here uses the Kontsevich-Zorich cocycle in a similar way to \cite{eskin2018algebraichull}. Our part \ref{part: no section1} is an application of their computations of an algebraic hull.
For part \ref{part: equidistribution1} we show an equidistribute result and adapt \cite{eskin2015isolation} to the high critical exponent setting. For this, we use $\varepsilon$-additive Margulis function theory that was developed \cite[\S5.1]{solan2024critical}.

The impossible section method does not appear to work for certain orbits, such as those in $\cH(2)$. However, in $\cH(2)$ McMullen \cite[Thm 5.8]{mcmullen2003genus2dynamics} shows that if an orbit $\SL_2(\RR).x$ has a hyperbolic element in its stabilizer, then it is periodic. 
This phenomenon extends to a dichotomy: either we can apply the impossible section method to an orbit, or McMullen's method shows that it must be periodic. 

\subsection{Structure of the paper}  
In Section \ref{sec: notations and preliminaries}, we recall standard notions on \teich\ dynamics. 
In Section \ref{sec: isolation}, we prove the equidistribution required by Part \ref{part: equidistribution1}. 
In Section \ref{sec: end of proof}, we complete the proof of Theorem \ref{thm: main}. 
The proof is divided into three parts: in Subsection \ref{ssec: work in H11}, we demonstrate how the method works for orbits in $\cH(1,1)$; 
in Subsection \ref{ssec: proof M large}, we extend the proof to a larger class of orbits; 
and in Subsection \ref{ssec: arithmeticity}, we apply different methods to the remaining orbits.

\subsection{Acknowledgements}
I thank my advisor Elon Lindenstrauss for introducing me to this problem, and Barak Weiss, Curtis McMullen, John Rached, and Alex Wright for useful discussions.
\section{Notation and preliminaries}
\label{sec: notations and preliminaries}
We first recall the dynamical system we are studying. 
\begin{definition}[\teich\ dynamics]\label{def: teich dynamics}
  Let $g\ge 2$.
  The \emph{moduli space of flat surfaces} $\tilde \cH_g$ is the collection of pairs $(S,\omega)$ of a genus $g$ Riemannian $S$ and a nonzero closed form $\omega\in \Omega^1(S;\CC)$ that is holomorphic with respect to some holomorphic structure of $S$, up to diffeomorphisms of the pair $(S,\omega)$. 
  Identifying $\CC\cong \RR^2$, the group $\GL_2(\RR)$ acts on the values of $\omega$ via this identification, thereby inducing an action on $\tilde\cH_g$. Note that the action does not preserve the holomorphic structure of $S$ for which $\omega$ is holomorphic. 
  
  We can decompose $\tilde\cH_g$ to stratas via the root pattern of $\omega$:
  let $k_1,\dots,k_m$ be positive integers satisfying $k_1 + \dots + k_m = g$. 
  The strata $\tilde\cH(k_1,\dots,k_m)\subset \tilde\cH_g$ consists of pairs $(S, \omega)$ such that $\omega$ vanishes at exactly $m$ points with orders $k_1, \dots, k_m$. 

  Since we consider the $\SL_2(\RR)$ action, the volume $\frac{i}{2}\int_S \omega\wedge \bar \omega$ is invariant to the action. 
  Hence we study the collection of volume one surfaces, 
  \[\cH_g = \left\{(S,\omega) \in \tilde\cH_g: \frac{i}{2}\int_S \omega\wedge \bar \omega=1\right\},\]
  and in this space we study the stratas $\cH(k_1,\dots,k_m) = \tilde \cH(k_1,\dots,k_m) \cap \cH_g$.
  
  We recall several notions on the spaces $\tilde \cH(k_1,\dots,k_m)$. 
  For every $(S,\omega) \in \tilde \cH(k_1,\dots,k_m)$ there is a neighborhood of $(S,\omega)$ locally parametrized by $H^1_{\rel}(S, \Sigma;\RR)\otimes \RR^2$, where $\Sigma$ is the set of zeroes of $\omega$ and $H^1_{\rel}(S, \Sigma;\RR)$ is the the cohomology group of $S$ relative to the set $\Sigma$ of zeroes of $\omega$. 
  
  Let $\tilde \cM\subseteq \tilde \cH(k_1,\dots,k_m)$ be a $\GL_2(\RR)$-invariant subset that is locally define by linear equations via the $H^1(S, \Sigma;\RR)\otimes \RR^2$-parametrization. 
  The intersection $\cM = \tilde \cM\cap \cH^1(k_1,\dots,k_m)$ is called affine manifold.
  Each affine manifold has an $\SL_2(\RR)$-invariant Lebesgue probability measure $\lambda_\cM$ defined via its locally affine structure. 
\end{definition}

\begin{definition}[Measure notaion]\label{def: notation measures}
  Let $X$ be a locally compact topological space. We denote by $\Prob(X)$ the set of probability measures on $X$.

\end{definition}
Finally we recall some notions in $\SL_2(\RR)$-dynamics. 
\begin{definition}[Critical exponent]\label{def: critical exponent}
  Let $\Lambda < \SL_2(\RR)$ be a discrete subgroup. Define the \emph{critical exponent} of $\Lambda$ by
  \[\delta(\Lambda) = \limsup_{R\to \infty}\frac{\log \#(B_{\SL_2(\RR)}(R)\cap \Lambda)}{R}.\]
  Here $B_{\SL_2(\RR)}(R)$ is the radius $R$ ball in $\SL_2(\RR)$ with respect to the Riemannian metric that is left invariant and right $\SO(2)$-invariant.
  If $\Lambda$ is a lattice, then $\delta(\Lambda) = 1$, and this is the maximal possible critical exponent.
\end{definition}
\begin{definition}\label{def: leafwise dimension}
  Let $B<\SL_2(\RR)$ denote the group generated by $\ta(t), \tu(s)$ for all $t,s\in \RR$. 
  For $X$ be an LCSC space and $B\acts X$ be a continuous action. 
  We say that a measure $\mu\in \Prob(X)$ is \emph{$\tu$-free} if for $\mu$-almost every $x\in X$ we have $\stab_\tu(x) = \{0\}$. 
  If in addition, $\mu$ is $\ta$-invariant and ergodic, one can define the \emph{leafwise dimension} quantity, $\dim^\tu(\mu) \in [0,1]$ as in \cite[Def. 3.7]{solan2024critical}. 
  In this work, we will only use properties of $\dim^\tu(\mu)$ that were established in \cite{solan2024critical}. 
\end{definition}
\section{Non-degenerate limits}
\label{sec: isolation}
In this section we will prove the following theorem, that is parallel to \cite[Lemma 4.7]{solan2024critical}.
\begin{theorem}\label{thm: limit is Lebesgue}
  Let $\vec k = (k_i)_{i=1}^m$ be an vector of positive integers. 
  Let $\cM\subseteq\cH(\vec k)$ be an affine manifold, and $(\mu_n)_{n=0}^\infty$ be a sequence of $\ta$-invariant and ergodic probability measures on $\cM$ such that $\mu_n(\cM') = 0$ for every affine submanifold $\cM' \subsetneq \cM$ and for every $n$ sufficiently large as a function of $\cM'$. 
  Assume that $\dim^\tu(\mu_n)
  \xrightarrow{n\to \infty}1$.
  Then $\mu_n \xrightarrow[\text{weak-}*]{n\to \infty} \lambda_{\cM}$. 
\end{theorem}
\subsection{Preliminary results}
In this subsection we gather several known results that will help us to prove Theorem \ref{thm: limit is Lebesgue}.
Fix $\vec k = (k_i)_{i=1}^m$. 
First we will introduce result on \teich\ dynamics.
\begin{theorem}[Eskin Mirzakhani \cite{eskin2018invariant}]\label{thm: Eskin Mirzakhani}
  Any $B$-invariant probability measure $\mu \in \Prob(\cH(\vec k))$ is a sum $\sum_{i} \alpha_i \lambda_{\cM_i}$, where $\cM_i$ are affine manifold in $\cH(\vec k)$. 
\end{theorem}
Set \begin{align*}
  r_\theta = \begin{pmatrix}
    \cos(\theta)&-\sin(\theta)\\
    \sin(\theta)&\cos(\theta)
  \end{pmatrix}.
\end{align*}
\begin{proposition}[Eskin Mirzakhani Mohammadi {{{\cite[Proposition 2.13]{mohammadi2023isolations}}}}]\label{prop: Eskin Mirzakhani Mohammadi technical}
  For every affine submanifold $\cM \subsetneq \cH(\vec k)$, there is a function $f_{\cM}:\cH(\vec k)\to [1,\infty]$ that satisfies the following properties.
  \begin{enumerate}
    \item $f_{\cM}$ is $SO(2)$-invariant.
    \item \label{point: function proper}$f_{\cM}^{-1}(\infty) = \cM$ and for every $\ell>0$ the set $\overline{f_{\cM}^{-1}([1,\ell])}$ is a compact subset of $\cH(\vec k)\setminus \cM$. 
    \item \label{point:Margulis}There exists $b>0$ depending on $\cM$ such that for every $0<c<1$ there is $t_c>0$, such that for every $t\ge t_c$ we have 
    \begin{align}
      \label{eq: strict Margulis}
      \frac{1}{2\pi}\int_{0}^{2\pi} f(\ta(t)r_{\theta}.x)\bd \theta \le c f_{\cM}(x) + b. 
    \end{align}
    \item \label{point: Lipschitz}For every compact subset $K\subseteq G$ there is $\sigma_K>1$ so that 
    \begin{align}
      \label{eq: Lipschitz}
      \forall g\in K,\quad f_{\cM}(gx) \in [\sigma_K^{-1} f_{\cM}(x), \sigma_Kf_{\cM}(x)].
    \end{align}
  \end{enumerate}
\end{proposition}

Next, we require several results from Solan \cite{solan2024critical} that will help us to deal with measures with high leafwise dimension.
Let $B$ as in Definition \ref{def: leafwise dimension}.
\begin{theorem}[Solan {{\cite[Thm. 3.9]{solan2024critical}}}]\label{thm: solan gap main}
  Let $B\acts X$ be a continuous action on an LCSC space $X$. 
  Let $(\mu_n)_{n=1}^\infty$ be a sequence of probability measures on $X$ that are $\tu$-free and $\ta$-invariant and ergodic.  
  Then if $(\mu_n)_{n=1}^\infty$ has a weak-$*$ limit, then it is $\tu$-invariant. 
\end{theorem}
\begin{definition}
  Let $B\acts X$ be a continuous action on an LCSC space $X$, and $\mu\in \Prob(X)$. 
  Denote by $S_\mu = \int_{-1}^0 \tu(s).\mu \bd s$. 
\end{definition}
\begin{lemma}[Solan {{\cite[Lem. 3.10]{solan2024critical}}}]\label{lem: solan quantitive leafwise entropy}
  Let $B\acts X$ be a continuous action on an LCSC space $X$, and $\mu\in \Prob(X)$ an $\ta$-invariant and ergodic measure. 
  Then there is a function $p:X\to [0,1]$ such that
  \begin{align}
    \label{eq: operator entropy}
    \int_{X} p(x) \log \frac{1}{p(x)} + (1-p(x))\log \frac{1}{1-p(x)} \bd S\mu(x) & = \dim^u(\mu)\log 2,   \\
    \label{eq: operator stationary}
    \int_{X} \omega_x \bd S\mu (x)                             & = S\mu,
  \end{align}
  where \begin{align}\label{eq: def omega x}
    \omega_x = p(x)\delta_{\ta(\log 2).x} + (1-p(x))\delta_{\tu(1)\ta(\log 2).x}.
  \end{align}
\end{lemma}
\begin{remark}[Interpretation of Lemma \ref{lem: solan quantitive leafwise entropy}]\label{rem: on Solan Markov chain}
    The map $x\mapsto \omega_x$ is a Markov chain, where the successor of $x$ distributes via $\omega_x$ and the measure $S\mu$ is stationary with respect to this Markov chain. 
\end{remark}

\begin{definition}[$(\varepsilon; T_0, T_1)$-additive Margulis function]\label{def: ep-additive Margulis function}
  Let $(X,\mu)$ be a probability measure space, $x\mapsto \nu_x$ a measurable map from $X$ to the space of probability measures on $X$ such that $\mu = \int_X\nu_x\bd \mu(x)$.
  In other words, $(X, x\mapsto \nu_x)$ is a Markov chain and $\mu$ is a stationary measure.
  A measurable function $\alpha: X \to [0,\infty)$ is called \emph{$(\varepsilon; T_0, T_1)$-additive Margulis function} for some $T_1> T_0>0$ large and $\varepsilon>0$ if the following conditions hold:
  \begin{enumerate}[label=M-\alph*), ref=(M-\alph*)]
    \item \label{cond: Lipshitz}For $\mu$-almost all $x\in X$, and for $\nu_x$-almost all $y\in X$, we have $\alpha(y) \in \alpha(x) + [-T_1, T_1]$.
    \item \label{cond: decay on averege}
          \begin{align}\label{eq: Margulis bound}
            \mu\left( \left\{ x\in X: T_1 \le \alpha(x) < T_0 + \int_X\alpha(y)\bd \nu_x(y)\right\} \right)<\varepsilon.
          \end{align}
  \end{enumerate}
\end{definition}
\begin{lemma}[Solan {{\cite[Lem. 5.3]{solan2024critical}}}]\label{lem: Margulis function goes down}
  In the setting of Definition \ref{def: ep-additive Margulis function},
  \begin{align*}
    \mu(\{x\in X:\alpha(x)\ge t\}) \le \frac{1}{\log \lfloor t / T_1\rfloor - 1} + \frac{T_0+T_1}{T_0}\varepsilon,
  \end{align*}
  for all $t\ge 3T_1$.
\end{lemma}

\subsection{Proof of Theorem \ref{thm: limit is Lebesgue}}
  Without loss of generality restrict $\mu_n$ to a subsequence such that the sequence has a weak-$*$ limit $\mu_\infty$. 
  Since no point in $\cH(\vec k)$ is invariant under the $\tu$-action, we deduce that every $\mu_n$ is $\tu$-free.
  By Theorem \ref{thm: solan gap main}, we deduce that $\mu_\infty$ is $\tu$-invariant. Hence, by Theorem \ref{thm: Eskin Mirzakhani}, $\mu_\infty = \sum_i \alpha_i \lambda_{\cM_i}$ for some collection $(\cM_i)_i$ of affine submanifolds of $\cM$ and $(\alpha_i)_i$ with $\sum_i \alpha_i \le 1$. 
  Assume to the contrary that this ergodic decomposition has some $\cM_i\subsetneq \cM$ with $\alpha_i > 0$. 
  
  Let $f_{\cM_i}:\cM\to [1,\infty]$ as in Proposition \ref{prop: Eskin Mirzakhani Mohammadi technical}.
  We wish to relate Eq. \eqref{eq: strict Margulis} to the Markov chain used in Lemma \ref{lem: solan quantitive leafwise entropy}. 
  \begin{claim}\label{claim: strict Margulis different op}
    Let $\cM\subsetneq \cH(\vec k)$ be an affine manifold and $f_\cM$ as in Proposition \ref{prop: Eskin Mirzakhani Mohammadi technical}. Then there is $b'$ such that for every $c\in (0,1)$ there is $\ell_c>0$ such that for every $\ell\ge \ell_c$ we have 
    \begin{align}\label{eq: strict Margulis different op}
      \frac{1}{2^\ell}\sum_{i=0}^{2^\ell-1} f_\cM(\tu(i)\ta(\ell\log 2).x) \le c f_\cM(x) + b'. 
    \end{align}
  \end{claim}
  \begin{proof}
    Fix $x\in \cH(\vec k)$. 
    Denote by 
    \begin{align*}
      S_0(x,\ell) &= \frac{1}{2^\ell}\sum_{i=0}^{2^\ell-1} f_\cM(\tu(i)\ta(\ell\log 2).x), \qquad \text{and}\\
      S_1(x,\ell) &= \frac{1}{2^\ell}\sum_{i=0}^{2^\ell-1} \int_0^1 f_\cM(\tu(i+s)\ta(\ell\log 2).x)\bd s = \int_0^1 f_\cM(\ta(\ell\log 2)\tu(s).x) \bd s.
    \end{align*}
    Let $K_1 = \{\tu(s): s\in [0,1]\}$. 
    By Eq. \eqref{eq: Lipschitz}, we obtain $S_0(x,\ell) \le \sigma_{K_1}S_1(s)$. 
    For every $s\in [0,1]$ denote $\beta(s) = -\arctan s$. 
    Then $L(s) := \tu(s)r_{-\beta(s)}$ is lower triangular. 
    Note that 
    \[\ta(\ell\log 2)\tu(s).x 
    = \ta(\ell\log 2)L(s)r_{\beta(s)}.x 
    = (\ta(\ell\log 2)L(s)\ta(-\ell\log 2))\ta(\ell\log 2)r_{\beta(s)}.x\]
    Let $K_2 = \overline{\{\ta(t)L(s)\ta(-t):s\in [0,1], t\ge 0\}}$, and note that it is compact and contains $\ta(\ell\log 2)L(s)\ta(-\ell\log 2)$. 
    Hence 
    \begin{align*}
      S_1(x,\ell) &= 
      \int_0^1 f_\cM(\ta(\ell\log 2)\tu(s).x) \bd s \\&= 
      \int_0^1 f_\cM((\ta(\ell\log 2)L(s)\ta(-\ell\log 2))\ta(\ell\log 2)r_{\beta(s)}.x) \bd s 
      \\&\le \sigma_{K_2}\int_0^1 f_\cM(\ta(\ell\log 2)r_{\beta(s)}.x)\bd s
    \end{align*}
    Note that $\sup_{[0,1]}\frac{1}{\arctan'} = 2$. Then 
    \begin{align*}
      \int_0^1 f_\cM(\ta(\ell\log 2)r_{\beta(s)}.x)\bd s 
    &\stackrel{\theta=\beta(s)}{=} \int_{\arctan 0}^{\arctan 1} f_\cM(\ta(\ell\log 2)r_{\theta}.x) \frac{1}{\arctan'(\tan\theta)}\bd \theta
    \\&\le 2 \int_0^{2\pi} f_\cM(\ta(\ell\log 2)r_{\theta}.x) \bd \theta.
    \end{align*}
    Altogether, if we set $C = 4\pi\sigma_{K_1}\sigma_{K_2}$, we obtain 
    \[S_0(x,\ell) \le \frac{C}{2\pi}\int_{0}^{2\pi} f_\cM(\ta(\ell\log 2)r_{\theta}.x)\bd \theta.\]
    The desired Inequality \eqref{eq: strict Margulis different op} follows with $n_c = t_{c/C}/\log 2$ and $b' = Cb$, where $t_{c/C}$ and $b$ are as in \ref{prop: Eskin Mirzakhani Mohammadi technical} Point \ref{point:Margulis}.
  \end{proof}
  We will now relate the Markov operator studied in Claim \ref{claim: strict Margulis different op} to the one provided by Lemma \ref{lem: solan quantitive leafwise entropy}. 
  Let $n$ be a sufficiently large index such that $\mu_n(\cM_i) = 0$. Let $\ell = \ell_{1/8} + 1$ be a large integer and $\delta>0$ be a small real number to be determined later.
  Apply Lemma \ref{lem: solan quantitive leafwise entropy} to $\mu_n$ and obtain a Markov chain 
  $x\mapsto \omega_{n,x}$.
  One can iterate this Markov chain and obtain a new Markov chain \[x\mapsto \omega_{n,x}^\ell = \sum_{i=0}^{2^\ell-1} p_{n,i}^\ell(x) \delta_{\tu(i)\ta(\ell\log 2).x},\]
  where $\sum_{i=0}^{2^\ell-1}p_{n,i}^\ell(x) = 1$, see \cite[Def. 5.11]{solan2024critical} for more details. 
  The entropy condition of Lemma \ref{lem: solan quantitive leafwise entropy}
  implies that \[\int_{\cM\setminus \cM_i} \sum_{i=0}^{2^\ell-1} p_{n,i}^\ell(x) \log \frac{1}{p_{n,i}^\ell(x)} \bd S\mu_n(x) = \ell \dim^u(\mu_n)\log 2.\]
  Hence (see \cite[Obs. 5.12]{solan2024critical} for more details), we deduce that if $n$ is sufficiently big such that $\dim^u(\mu_n)$ is sufficiently close to $1$, there is a set $X_{\rm good}^{n,\ell}\subseteq \cM\setminus \cM_i$ such that $S\mu_n(X_{\rm good}^{n,\ell}) \ge 1-\delta$ and for every $x\in X_{\rm good}^{n,\ell}$ and for every $i=0,1,\dots,2^{\ell}-1$ we have $\left| p_{n,i}^\ell(x) - 2^{-\ell}\right| \le \delta$. 

  We will use the machinery of $(\varepsilon; T_0, T_1)$-additive Margulis function \ref{def: ep-additive Margulis function}. 
  The measure space is $(\cM \setminus \cM_i, S\mu_n)$. 
  The Markov chain is $x\mapsto \omega_{n,x}^\ell$, and it preserves $S\mu_n$. The additive Margulis function is $\alpha = \log \circ  f_{\cM_i}$. 
  We will find $\varepsilon, T_0, T_1$ such that $\alpha$ is $\varepsilon, T_0, T_1$-additive Margulis function. 
  Indeed, if $T_1 \ge \log \sigma_{\{\ta(\ell\log 2)\tu(s): s\in [0,1]\}}$ then Condition \ref{cond: Lipshitz} holds, where $\sigma_\bullet$ is as in Proposition \ref{prop: Eskin Mirzakhani Mohammadi technical} Point \ref{point: Lipschitz}. 
  To show Condition \ref{cond: decay on averege}, let $x\in X_{\rm good}^{n,\ell}$, with $\alpha(x) > T_1$. 
  Then 
  \begin{align*}
    \int_X\alpha(y)\bd \omega_{n,x}^\ell(y)
    &= 
    \sum_{i=0}^{2^\ell-1}p_{n}^\ell(x)\alpha(\tu(i)\ta(\ell\log 2).x)
    \\&\leftstackrel{x\in X_{\rm good}^{n,\ell}}\le
    \sum_{i=0}^{2^\ell-1}2^{-\ell}(x)\alpha(\tu(i)\ta(\ell\log 2).x) + 2^\ell\delta T_1
    \\&\leftstackrel{\text{log is concave}}\le
    \log\left(\sum_{i=0}^{2^\ell-1}2^{-\ell}(x)f_{\cM_i}(\tu(i)\ta(\ell\log 2).x)\right) + 2^\ell\delta T_1
    \\&\leftstackrel{\ref{claim: strict Margulis different op} + \ell>\ell_{1/8}}\le
    \log\left(\frac{1}{8} f_{\cM_i}(x) + b'\right) + 2^\ell\delta T_1
    \\&\le 
    \log\left(\frac{1}{8} f_{\cM_i}(x)(1+b'/T_1)\right) + 2^\ell\delta T_1
    \\&=\alpha(x) + \log \frac{1}{8} + \log (1+b'/T_1) + 2^\ell\delta T_1.
  \end{align*}
  Set \[T_1 = \max(b', \log \sigma_{\{\ta(\ell\log 2)\tu(s): s\in [0,1]\}}, \log 2).\] 
  Then 
  \[\int_X\alpha(y)\bd \omega_{n,x}^\ell(y) \le \alpha(x) - 2\log 2 + 2^\ell\delta T_1. \]
  Assume that $\delta \le \frac{\log 2}{2^\ell T_1}$. 
  Then we obtain 
  \[\int_X\alpha(y)\bd \omega_{n,x}^\ell(y) \le \alpha(x) - \log 2.\]
  Setting $T_0 = \log 2$ ensures that the set estimated in Eq. \eqref{eq: Margulis bound} is disjoint from $X_{\rm good}^{n,\ell}$, and hence its measure is at most $\delta$. 
  Altogether, we have proved that $\alpha$ is a $(\delta; \log 2, T_1)$-additive Margulis function. 
  
  We can now apply Claim \ref{lem: Margulis function goes down}, and obtain that 
  \[S\mu(\{x\in \cM \setminus \cM_i:\alpha(x)\ge t\})\le \frac{1}{\log \lfloor t / T_1\rfloor - 1} + \frac{\log2 +T_1}{\log 2}\delta,\]
  for every $t\ge 0$.

  Taking $\delta>0$ sufficiently small and $t>0$ sufficiently large, we obtain that 
  \[S\mu_n (\{x\in \cM\setminus \cM_i: \alpha(x) \ge t\}) \le  \alpha_i/2.\]
  Hence 
  \[S\mu_n (\overline{\{x\in \cM\setminus \cM_i: \alpha(x) \le t\}}) \ge 1- \alpha_i/2.\]
  By Proposition \ref{prop: Eskin Mirzakhani Mohammadi technical} Point \ref{point: function proper} the set  \[\{x\in \cM\setminus \cM_i: \alpha(x) \le t\} = \{x\in \cM: f_{\cM_i}(x) \le e^t\},\] is compact.
  However, this contradicts the limit $\mu_\infty = S\mu_\infty = \lim_{n\to \infty}S\mu_n$ giving $\cM_i$ measure $\alpha_i$. 
  Hence $\mu_\infty$ has no components that are the Lebesgue measure on an affine manifold strictly contained in $\cM$, and hence $\mu_\infty =\alpha_0\lambda_{\cM}$. 
  If $\alpha_0 < 1$ we can run a similar argument with $\emptyset$ instead of $\cM_i$ and obtain a contradiction. 
  Hence $\mu_\infty = \lambda_{\cM}$, as desired.

\section{Completion of the proof}
\label{sec: end of proof}

  

In this section we prove Theorem \ref{thm: main}, using the Impossible Section Method \ref{prop: impossible section method1}. The section will be divided into $3$ main parts. In Subsection \ref{ssec: work in H11} we prove Theorem \ref{thm: main} in the strate $\cH(1,1)$. In the Subsection \ref{ssec: proof M large} we push the proof to a greater generality: we obtain conditions on orbits that enables us to use the impossible section method. These conditions may fail. However, precisely when these conditions fail, we show that if the stabilizer is Zariski dense, then it is a lattice. This is done in Subsection \ref{ssec: arithmeticity} using methods of McMullen \cite{mcmullen2003teichmuller,mcmullen2003genus2dynamics}. 

Let $g\ge 2$, and consider the action $\SL_2(\RR)\acts \cH_g$. 
Assume that $(\SL_2(\RR).x_i)_{i=1}^\infty$ is an infinite sequence of orbits, with $\delta(\stab_{\SL_2(\RR)}(x_i))\to 1$, and $\stab_{\SL_2(\RR)}(x_i)$ is never a lattice.  
Without loss of generality, we may assume that $\delta(\stab_{\SL_2(\RR)}(x_i))$ is always positive, and in particular, $\stab_{\SL_2(\RR)}(x_i)$ is Zariski dense. 
Let $\cM$ be a minimal affine manifold that contains infinitely many $x_i$. Restrict to subsequence and assume that $x_i\in \cM$ for all $i\ge 1$. 
We may proceed with the Impossible Section Method \ref{prop: impossible section method1}. 
For each $i\ge 1$ let $\mu_i \in \Prob(\SL_2(\RR).x_i)$ be the $\ta$-invariant and ergodic probability measure satisfying $h_{\mu_i}(\ta(1)) \to 1$. This collection of measures is provided by the Impossible Section Method \ref{prop: impossible section method1}. 
By \cite[Thm~7.6 (ii)]{einsiedler2010diagonal}, the entropy coincides with the leafwise dimension,
\begin{align}
  h_{\mu_i}(\ta(1)) = \dim^\tu(\mu_i) \to 1.
\end{align}
By Theorem \ref{thm: limit is Lebesgue}, we deduce that we have a weak-$*$ convergence to the affine probability measure $\lambda_\cM$. 
This completes Part \ref{part: equidistribution1} of the method.
It is left to prove Parts \ref{part: lifts to projective1} and \ref{part: no section1}. 
We begin by showing an explicit argument for $g=2$ and $\cM\subseteq \cH^1(1,1)$.

\subsection{Proof in the strata \texorpdfstring{$\cH(1,1)$}{H(1,1)}}
\label{ssec: work in H11}
We will work with the following bundles on $\cH^1(1,1)$:
\begin{definition}
  Let $x = (S,\omega) \in \cH^1(1,1)$ be a flat surface. 
  We can consider the homology group $H_1(x;\ZZ) := H_1(S;\ZZ)$. 
  Varying $x$ defines a locally constant bundle $H_1(-;\ZZ)$ on $\cH^1(1,1)$. 
  Let $\Sigma\subseteq S$ denote the set of zeros of $\omega$. 
  One can consider the relative homology $H_1^\rel(x;\ZZ) = H_1^\rel(S,\Sigma;\ZZ)$. 
  Varying $x$ defines another locally constant bundle $H_1^\rel(-;\ZZ)$ on $\cH^1(1,1)$. 
  We will also consider both homology groups with coefficients in $\QQ$ and in $\RR$ as well.
  Let $p^*:H_1(x; \ZZ) \to H_1^\rel(x; \ZZ)$ denote the inclusion. Varying $x$ defines an inclusion of bundles
  \begin{align}
    \label{eq: short exact seq}
    \xymatrix{H_1(-;\ZZ) \ar@{^{(}->}[r]^{p^*} & H_1^\rel(-;\ZZ).}
  \end{align}
  Let $x\in \cH(1,1)$ and $g\in \SL_2(\RR)$. The definition of the action $g.x$ we obtain an homeomorphism of the underlying spaces of $x$ and $g.x$.
  This give a bijection $g_*:H_1(x;\ZZ)\to H_1(g.x;\ZZ)$, and a similar definition for the relative homology and other coefficient rings.

  Define a map $\tau_x: H_1^\rel(x; \RR)\to \RR^2$ by sending a class $[\alpha]$, where either $\alpha:[0,1]\to S$ is either a loop or $\partial\alpha \subseteq \Sigma$, to $\int_\alpha \omega$. 
  The map $\tau_x$ satisfies that for every $\alpha\in H_1^\rel(x; \RR)$ and $g\in \SL_2(\RR)$ we have $\tau_{g.x}(g_* \alpha) = g.\tau_x(\alpha)$, or in a diagram, the following commutes.
  \begin{align}\label{eq: g equivarince of pi}
    \xymatrix{H_1^\rel(x; \RR)\ar[r]^{g_*}\ar[d]^{\tau_x} &H_1^\rel(g.x; \RR)\ar[d]^{\tau_{g.x}} \\
              \RR^2\ar[r]^{g.-}&\RR^2.
    }
  \end{align}
  Recall the exact sequence of relative homology
  \begin{align*}
    \xymatrix{H_1(S;\ZZ) \ar@{^{(}->}[r]^{p^*} & H_1^\rel(S, \Sigma;\ZZ)\ar[r]& H^0(\Sigma).}
  \end{align*}
  By replacing $\cH(1,1)$ by an the finite cover in which the zeroes $\Sigma$ of the form are labeled, we deduce that ${\rm coker}(p^*)$ is the constant bundle.
\end{definition}
\begin{obs}
  Let $a$ be as in Theorem \ref{thm: McMullen} and $K_a = \QQ[a]$ be a number field. 
  For every flat surface $S(a)$, the images $\tau_{(S(a), \bd z)}(H_1^\rel(S(a); \ZZ))$ and $\tau_{(S(a), \bd z)}(H_1(S(a); \ZZ))$ lie in $K_a^2$ and are  $\ZZ$-modules of rank $4$. 
  In particular, $\tau_{(S(a), \bd z)}(H_1^\rel((S(a), \bd z); \QQ)) = \tau(H_1(S(a); \QQ)) = K_a^2$.
\end{obs}
The same holds in more generality:
\begin{proposition}\label{prop: general tau eq}
  Let $x_0\in \cH^1(1,1)$ satisfy that $\stab_{\SL_2(\RR)}(x_0)$ is Zariski dense in $\SL_2(\RR)$ but is not a lattice. 
  Then \[\tau_{x_0}(H_1^\rel(x_0; \QQ)) = \tau_{x_0}(H_1(x_0; \QQ)) \subset \RR^2,\] is a $\QQ$-vector space of $\QQ$-dimension $4$.
\end{proposition}
\begin{proof}
  By McMullen \cite[Theorem 9.4]{mcmullen2003teichmuller}, we obtain that $\tau_{x_0}(H_1^\rel(x_0; \QQ)) = \tau_{x_0}(H_1(x_0; \QQ))$. 
  By \cite[Theorem 9.5]{mcmullen2003teichmuller}, this is an even dimensional $\QQ$ vector space. 
  Considering the edges in any presentation of $x_0$ as a polygon with glued edges we deduce that $\tau(H_1^\rel(x_0; \QQ)) \neq \{0\}$, and moreover, $\spa_{\RR}(\tau(H_1^\rel(x_0; \QQ))) = \RR^2$. Hence $\dim_\QQ \tau(H_1^\rel(x_0; \QQ)) \ge 2$.
  If $\dim_\QQ \tau(H_1^\rel(x_0; \QQ)) = 2$ then $\rk \tau(H_1^\rel(x_0; \ZZ))$ is a lattice in $\RR^2$. Therefore $x_0$ covers the torus $\RR^2 / H_1^\rel(x_0; \ZZ)$, such that the two zeros of the abelian form of $X$ project to a single point in the torus.
  However, this implies that $\stab_{\SL_2(\RR)}(x_0,\omega_0)$ is an arithmetic lattice, while we assumed it is not a lattice.
  Consequently, $\dim_\QQ \tau(H_1(x_0; \QQ)) = \dim_\QQ \tau(H_1^\rel(x_0; \QQ)) = 4$.
\end{proof}
\begin{corollary}\label{cor: splitting}
  Let $x_0\in \cH^1(1,1)$ satisfy that $\stab_{\SL_2(\RR)}(x_0)$ is Zariski dense in $\SL_2(\RR)$ but not a lattice. 
  When we restrict the embedding $p^*$ to the $\SL_2(\RR)$-orbit $\SL_2(\RR). x_0$ and take coefficients over $\RR$, the embedding has an invariant left inverse 
  \[r:H_1^\rel(-;\RR)|_{\SL_2(\RR).x_0}\to H_1(-;\RR)|_{\SL_2(\RR).x_0}, \quad\text{with}\quad r\circ p^*|_{\SL_2(\RR).x_0} = {\rm Id}_{H_1(-;\RR)|_{\SL_2(\RR).x_0}}.\]
\end{corollary}
\begin{proof}
  Since $\dim_\QQ(\tau_{x_0}(H_1(x_0; \QQ))) = 4$ we obtain that \[\tau_{x_0}|_{H_1(x_0; \QQ)}:H_1(x_0; \QQ)\to \tau_{x_0}(H_1(x_0; \QQ))\] is an isomorphism. 
  Define the left inverse $r_\QQ: H_1^\rel(x_0; \QQ)\to H_1(x_0; \QQ)$ by $r_\QQ = \tau_{x_0}|_{H_1(x_0; \QQ)}^{-1} \circ \tau_{x_0}|_{H_1^{\rel}(x_0; \QQ)}$. 
  The natural way in which $r_\QQ$ was defined may be applied for every $x_1\in \SL_2(\RR).x_0$, and hence it extends to an invariant left inverse 
  \[r_\QQ:H_1^\rel(-;\QQ)|_{\SL_2(\RR). x_0}\to H_1(-;\QQ)|_{\SL_2(\RR). x_0}\quad \text{with}\quad r_\QQ \circ p^*|_{\SL_2(\RR). x_0} = {\rm Id}_{H_1(-;\RR)|_{\SL_2(\RR). x_0}}.\]
  Extension of scalars to $\RR$ yields the desired left inverse.
\end{proof}
\begin{definition}
  Consider the locally constant projective bundle \[P = \bigsqcup_{x_0\in \cH^1(1,1)}\PP(H_1^\rel(x_0;\RR)).\] 
  Let $\pi:P\to \cH^1(1,1)$ denote the projection and $\SL_2(\RR).x_0$ an orbit such that $\stab_{\SL_2(\RR)}(x_0)$ is Zariski dense in $\SL_2(\RR)$ but not a lattice. 
  We obtain an invariant section $s_{\SL_2(\RR).x_0}: \SL_2(\RR).x_0\to P$, defined by $s_{\SL_2(\RR).x_0}(x_1) = \ker r|_{x_1}$ for every $x_1\in \SL_2(\RR).x_0$, where $r$ is the left inverse obtained by Corollary \ref{cor: splitting}. 
\end{definition}
This concludes \ref{part: lifts to projective1}: Corollary \ref{cor: splitting} introduces the projective bundle \[P = \bigsqcup_{(S,\omega)\in \cH(1,1)}\PP(H^\rel_1(S;\RR)),\] and locally constant sections $s_i$ of $P$ on each orbit $\SL_2(\RR).x_i$. Restricting $P$ to $\cM$ we obtain the desired bundle $P_\cM$ and the desired sections required by Part \ref{part: lifts to projective1}.

Finally, we wish to show that there is no invariant probability measure on $P_\cM$ projecting to $\lambda_\cM$ on $\cM$. 
Recall McMullen's classification \cite{mcmullen2003genus2dynamics} of affine manifolds in $\cH^1(1,1)$:
\begin{theorem}\label{thm: mcmullen genus 2 classification}
  The connected affine submanifolds of $\cH^1(1,1)$ are 
  \begin{enumerate}
    \item Periodic $\SL_2(\RR)$ orbits.
    \item Let $\cO$ be a quadratic order. 
    The manifold $\cM_\cO$ consisting of curves $(S,\omega)$ such that $Jac(S;\omega)$ has real multiplication in $\cO$ (see McMullen \cite[Thms. 5.1, 5.5]{mcmullen2003genus2dynamics} for more properties of orders curves with Jacobian of real multiplication).
    \item The manifold $\cH(1,1)$. 
  \end{enumerate}
\end{theorem}
Since no point $x_i$ is contained in a periodic orbit, one of the following holds:
\begin{enumerate}
  \item 
  There may be a quadratic order $\cO$ in a quadratic number field $K$ such that $\cM = \cM_\cO$. 
  Theorem \ref{thm: mcmullen genus 2 classification} allows also orders $\cO\subset \QQ\times \QQ$. However, using \cite[Thm. 5.5]{mcmullen2003genus2dynamics} we deduce that if this holds, then $\tau_{x_0}(H_1(x_0; \QQ))$ is two dimensional and this contradict Claim \ref{prop: general tau eq}. 
  \item 
  $\cM = \cH(1,1)$. 
\end{enumerate}

We will use the notion of the algebraic hull of a vector bundle, as defined in \cite{eskin2018algebraichull}. 
We will not define it here, only declare its type and mention a few properties that follows from its definition:
\begin{claim}[Properties of the algebraic hull]\label{claim: prop of alg hull}
  For every group action $G\acts X$ of a Borel group and a borel space $X$, a $G$-invariant measure $\mu$ on $X$, and a finite dimensional real vector bundle $V$ on $X$ on which $G$ acts as well, the \emph{(measurable) algebraic hull} of the vector bundle is a measurable map that attaches for $\mu$-almost every every $x\in X$, an algebraic subgroup $A_{V,x}\subseteq \Aut(V_x)$. It obeys the following rules:
  \begin{enumerate}
    \item For $\mu$-almost all $x\in X$, and every $g\in G$, we have $A_{V,g.x} = gA_{V,x}g^{-1}$. 
    \item For $\mu$-almost all $x\in X$, the algebraic hull of the dual bundle is the dual of algebraic hull of the bundle, $A_{V^*,x} = (A_{V,x})^*$.
    \item If $V$ contains an invariant vector bundle $W<V$, then for $\mu$-almost every $x$,
    \begin{itemize}
      \item $A_{V,x}$ preserves $W_x$ 
      \item $A_{W,x}$ is the image of the restriction to $W_x$ of $A_{V,x}$,
      \item $A_{V/W,x}$ is the image of the induced action of $A_{V,x}$ on $V_x/W_x$.
    \end{itemize}
    \item If $\mu$ is $G$-ergodic and $V$ has a constant dimension $d$ then here exists a subgroup $A_V\subseteq \GL_d(\RR)$ that is conjugated to almost all $A_{V,x}$. In addition there is a measurable trivialization $\varphi_x:V_x\to \RR^d$ so that for $\mu$-almost every $x\in X$ and every $g\in G$ the composition 
    \[\RR^d\xrightarrow{\varphi_x^{-1}}V_x \xrightarrow{g\cdot -}V_{g.x}\xrightarrow{\varphi_{g.x}}\RR^d,\]
    lies in $A_V$. 
  \end{enumerate}
\end{claim}

The following claim is an application of \cite[Theorem 1.2]{eskin2018algebraichull}.
\begin{claim}\label{claim: algebraic hull computation}
  With the $\SL_2(\RR)$ action on $(\cM, \lambda_\cM)$,
  the algebraic hull of $H_1^\rel(-, \RR)$, namely $A_{H_1^\rel, x}$
  conjugate to 
  \[\begin{pmatrix}
    * & *&0&0&0\\
    * & *&0&0&0\\
    0&0&* & *&*\\
    0&0&* & *&*\\
    0&0&0&0&1
  \end{pmatrix},\]
  for $\lambda_{\cM}$-almost all $x \in \cM$, where the determinants of the two $2\times 2$ blocks are $1$. 
\end{claim}
\begin{proof}
  If $\cM = \cH^1(1,1)$, then the claim is a direct application of \cite[Theorem 1.2]{eskin2018algebraichull} to $\cM$, and the duality $H_1^\rel(-, \RR) = (H^1_\rel(-, \RR))^*$. 

  Suppose now that $\cM = \cM_\cO$ for some order in a quadratic number field $\cO\subset K\subset \RR$. 
  For the rest of the proof, the bundles and algebraic hulls are always restricted to $\cM_\cO$. 
  Let $p:H^1_\rel \to H^1$ be the map of bundles defined as in \cite[Theorem 1.2]{eskin2018algebraichull}, dual to the map $p^*$ we defined here. 
  Let $T = (\ker \tau)^\perp\subseteq H^1_\rel(-, \RR)$ be the tautological plane.
  By \cite[Theorem 5.1 Point 3]{mcmullen2003genus2dynamics}, the group $p(T)\subseteq H^1(-, \RR)$ is defined over $K$ via the rational structure $H^1(-, \ZZ)$ of $H^1(-, \RR)$.
  This implies that one can apply to $p(T)$ the Galois automorphisms $\Gal(K/\QQ)$. Let $\sigma$ denote the nontrivial element in $\Gal(K/\QQ)$. \cite[Theorem 5.1 Point 3]{mcmullen2003genus2dynamics} also states that $p(T)^\perp = p(T)^\sigma$, where $p(T)^\perp$ is the orthogonal complement with respect to the symplectic cup product. 
  Then we have $H^1(-, \RR) = p(T) \oplus p(T)^\sigma$. 
  Hence $H^1_\rel(-, \RR) = T \oplus p^{-1}(p(T)^\sigma) = T \oplus p^{-1}(p(T))^\sigma$.
  Claim \ref{claim: prop of alg hull} yields that the algebraic hull $A_{H^1_\rel(-, \RR), -}$ embeds in 
  $A_{T, -}\times A_{p^{-1}(p(T))^\sigma, -}$, and projects onto the factors. 

  We now simplify the expression of the bundle $p^{-1}(p(T))^\sigma$. 
  By \cite[Theorem 5.1]{mcmullen2003genus2dynamics} again, we have $T\cM_\cO = p^{-1}(p(T))$.
  Hence $p^{-1}(p(T))^\sigma = (T\cM_\cO)^\sigma$.


  By \cite[Theorem 1.2]{eskin2018algebraichull}, the algebraic hull of $T\cM_\cO$ at a point $x$, is the collection of all automorphisms of $(T\cM_\cO)_x$ that fixes $\ker p$ and preserves the tautological plane $T$. In other words, it is conjugates to the set of matrices of the form 
  \[\begin{pmatrix}
    * & * & 0\\
    * & * & 0\\
    0 & 0 & 1
  \end{pmatrix},\]
  with determinant is $1$, where the first two coordinates are conjugate to $T$ and the last coordinate coordinate $\ker p$.
  In particular, for $\mu$-almost every $x$ we have $A_{T,x}\cong \SL(T_x)$. 

  Similarly, the algebraic hull of $T\cM_\cO^\sigma$ is is the collection of all automorphisms of $(T\cM_\cO)_x^\sigma$ satisfying the same conditions except of the requirement not to stabilize the tautological plane $T$. In other words, it is conjugates to the set of matrices of the form 
  \[G_3 = \left\{
    \begin{pmatrix}
    * & * & 0\\
    * & * & 0\\
    * & * & 1
    \end{pmatrix}
  \right\}
  ,\]
  with determinant is $1$, where the last coordinate represents $\ker p$.
  We denote this group of matrices by $\RR^2 \rtimes \SL_2(\RR)$. 
  Hence $A_{H^1_\rel} \subseteq A_T \times A_{T\cM_\cO^\sigma} \cong \SL_2(\RR)\times (\RR^2\rtimes \SL_2(\RR))$, and projects onto each coordinate. 
  One can see that (up to conjugations) there is only one connected proper algebraic subgroup of $\SL_2(\RR)\times (\RR^2\rtimes \SL_2(\RR))$ that projects onto the two coordinates, namely 
  \begin{align}\label{eq: impossible hull}
    A_{H^1_\rel}\cong \left\{
      \begin{pNiceArray}{cc|ccc}
        a & b & 0 & 0 & 0 \\
        c & d & 0 & 0 & 0 \\
        \hline
        0 & 0 & a & b & 0 \\
        0 & 0 & c & d & 0 \\
        0 & 0 & e & f & 1
      \end{pNiceArray} :\det\begin{pmatrix}
        a&b\\c&d
      \end{pmatrix} = 1, a,b,c,d,e,f \in \RR\right\}.
    \end{align}
  However, this implies that the two top Lyapunov exponents of the $\ta_t$ action are equal. 
  Since this contradicts Forni \cite[Cor. 2.2]{forni2002deviation}, we get a contradiction to $A_{H^1_\rel}\subsetneq A_T \times A_{T\cM_\cO^\sigma}$. 
  Hence $A_{H^1_\rel}\cong A_T \times A_{T\cM_\cO^\sigma}$.
  Duality shows that $A_{H_1^\rel}$ admits the desired description. 
\end{proof}
\begin{remark}
  The use of \cite[Cor. 2.2]{forni2002deviation} is given to give a complete picture. 
  We could finish the proof with a weaker result in Claim \ref{claim: algebraic hull computation}: either the algebraic hull is as in the claim, or it is the transpose of Eq. \eqref{eq: impossible hull}. 
\end{remark}

We will now deduce Part \ref{part: no section1} of Proposition \ref{prop: impossible section method1}. 
Assume to the contrary that there is an $\SL_2(\RR)$-invariant probability measure $\mu\in \Prob(P)$, so that $\pi_*\mu = \lambda_\cM$.
\begin{claim}\label{claim: no inv measure}
  Suppose $G$ acts on LCSC topological space $X$ persevering a probability measure $\mu\in X$.
  Let $V$ be a real vector space bundle over $X$ on which $G$ acts. 
  Assume that $\nu$ is a $G$ invariant probability measure on $\PP(V)$ that projects to $\mu$ on $Y$. 
  Disintegrate $\nu = \int \nu_x \bd \mu(x)$ where $\nu_x$ is supported on $\PP(V_x)$. 
  Then for $\mu$-almost all $x$ the algebraic hull $A_{V,x}$ preserves $\nu_x$. 
\end{claim}
The proof of the claim appears in \cite[page 298, last page of the first proof of Theorem 5.2]{eskin2018algebraichull}. 
We deduce that the disintegration $\mu = \int_\cM \mu^{x}\bd \lambda_\cM(x)$ satisfies that almost every $\mu^{x}$ is invariant to a group conjugate to $A_{H_1^\rel, x}$. 
However, by Claim \ref{claim: algebraic hull computation}, there are no $A_{H_1^\rel, x}$-invariant probability measures on $\PP(H_1^\rel(x;\RR))$. 
This contradicts the existence of an $\SL_2(\RR)$-invariant probability measure on $P$, and proves Part \ref{part: no section1} of Proposition \ref{prop: impossible section method1}. 

Now that we proved the three parts of Proposition \ref{prop: impossible section method1}, it contradicts the existence of a sequence $(x_i)_{i=1}^\infty \in \cH^1(1,1)$ such that $\stab_{\SL_2(\RR)}(x_i)$ is not a lattice but $\delta(\stab_{\SL_2(\RR)}(x_i))\to 1$. Hence there is an absolute bound $1-\varepsilon_2 < 1$ on the critical exponents $\delta(\stab_{\SL_2(\RR)}(x))$ for $x\in \cH(1,1)$ for which $\stab_{\SL_2(\RR)}(x)$ is not a lattice.

\subsection{Proof in the case that \texorpdfstring{$\cM$}{M} is `large'}
\label{ssec: proof M large}
In this section, we extend the proof from the case $\cM \subseteq \cH(1,1)$, discussed in the previous subsection, to a more general class of manifolds, that satisfies the \emph{large size assumption} and will be specified in Definition \ref{def: large size assumption}. The following subsection will address the complementary assumptions on $\cM$ with different techniques.

In this subsection we work with complex vector bundles instead of real ones. 
Denote by $Q$ the symplectic cup product on $H^1(-,\CC)$ and by $\tilde Q$ its pullback to $H_{\rel}^1(-,\CC)$. 
Recall \cite[Thm. 1.4]{avila2017symplectic} that $Q$ is nondegenerate on $p(T\cM)$, and in particular, $\dim p(T\cM)$ is even. Denote $\rk(\cM) = \dim(p(T\cM))/2$. 
Recall the \emph{field of definition} $K$ of $\cM$ as defined by Wright \cite{wright2014field}. It is shown that $T\cM\subseteq H^1_\rel(-,\CC)|_{\cM}$ is defined over $K$ in the $H^1_\rel(-,\ZZ)|_{\cM}$ integral structure. Denote $d = [K:\QQ]$. 
Hence there are $d$ complex embeddings $\Hom(K, \CC)$. For every $\sigma:K\to \CC$ we can consider the conjugation $T\cM^\sigma$. 
Wright \cite{wright2014field} shows in addition that the sum of the vector bundles $\{p(T\cM)^\sigma:\sigma\in \Hom(K,\RR)\}$ is a direct sum.

\begin{claim}\label{claim: direct sum}
  If $x_0\in \cM$ satisfies that $\stab_{\SL_2(\RR)}(x_0)$ contains a hyperbolic element, then
  $T_{x_0}$ is defined over a number field $K_0\supseteq K$,
  the sum of the vector spaces $(p(T_{x_0})^{\sigma'})_{\sigma':K_0\to \CC}$ is a direct sum,
  the symplectic form $\tilde Q$ is nondegenerate on each $T_{x_0}^{\sigma'}$ and they are pairwise orthogonal.
\end{claim}
\begin{proof}
  Since $\stab_{\SL_2(\RR)}(x_0)$ is discrete and Zariski dense, there is a hyperbolic element $g_0\in \stab_{\SL_2(\RR)}(x_0)$. 
  For every $\lambda$ and a linear operator $A$ on a space $V$ denote by $V^{A,\lambda}$ the $\lambda$ eigenspace of $A$. 
  By McMullen \cite[Thm. 5.3]{mcmullen2003billiards} (see also \cite[Cor. 2.2]{forni2002deviation}), the element $g_0$ has eigenvalues $\lambda_0, \lambda_0^{-1}$ with $\lambda_0>1$ and $\lambda_0$ is a multiplicity one eigenvalue of $(g_0)^*|_{H^1_{\rel}(x_0;\CC)}$, the action of $g_0$ on $H^1_{\rel}(x_0;\CC)$. 
  Since $(g_0)^*|_{H^1_{\rel}(x_0;\CC)}$ preserves the symplectic form on $H^1_{\rel}(x_0;\CC)$, we deduce that eigenvalues come in pairs of inverses and $\lambda_0^{-1}$ also has multiplicity $1$. 
  Since $(g_0)^*|_{H^1_{\rel}(x_0;\CC)}$ acts on $T_{x_0}$ via $g_0$ we deduce that 
  \begin{align}\label{eq: T as eigenspace}
    T_{x_0} = H^1_{\rel}(x_0;\CC)^{(g_0)^*, \lambda_0} \oplus H^1_{\rel}(x_0;\CC)^{(g_0)^*, \lambda_0^{-1}} = H^1_{\rel}(x_0;\CC)^{(g_0)^* + (g_0^{-1})^*,\lambda_0 + \lambda_0^{-1}}.
  \end{align}
  
  Since $\lambda_0$ is an eigenvalue of the integral operator, we deduce that $\lambda_0$ is an algebraic integer. 
  Set $K_0 = \QQ[\lambda_0 + \lambda_0^{-1}]$, and note that Eq. \eqref{eq: T as eigenspace} represents $T_{x_0}$ as a vector space defined over $K_0$. 
  For every ${\sigma'}\in \Hom(K_0, \CC)$, we can conjugate Eq. \eqref{eq: T as eigenspace} and obtain 
  \begin{align}\label{eq: T sig as eigenspace}
    T_{x_0}^{\sigma'} = H^1_{\rel}(x_0;\CC)^{(g_0)^* + (g_0^{-1})^*,{\sigma'}(\lambda_0 + \lambda_0^{-1})}.
  \end{align}
  Since the eigenspaces of an operator are summed directly, we deduce that $T_{x_0}^{\sigma'}$ are summed directly. Since $\tilde Q$ is nondegenerate on each $T_{x_0}^{\sigma'}$ and they are orthogonal, we obtain that $\tilde Q$ is nondegenerate on their direct sum $\bigoplus_{\sigma'\in \Hom(K_i, \CC)} T_{x_0}^{\sigma'}$. Since $\tilde Q$ factors through $p$, we obtain that $p$  is bijective on this direct sum, the image of the direct sum under $p$ is again a direct sum
  \[\bigoplus_{\sigma'\in \Hom(K_i, \CC)} p(T_{x_0})^{\sigma'}.\]

  For every ${\sigma'}:K_0\to \CC$ let $\tilde {\sigma}':\CC\to \CC$ be a field automorphism that extends ${\sigma'}$. 
  Then we obtain that \[T_{x_0}^{\sigma'} = H^1_{\rel}(x_0;\CC)^{(g_0)^*,\tilde {\sigma}'(\lambda_0)} \oplus H^1_{\rel}(x_0;\CC)^{(g_0)^*,\tilde {\sigma}'(\lambda_0^{-1})}.\]
  Since $(g_0)^*$ preserves $\tilde Q$, we deduce that two eigenvectors can be non-orthogonal only if the eigenvalues are inverses. 
  Therefore $(T_{x_0}^{\sigma'})_{{\sigma'}\in \Hom(K_0, \CC)}$ are mutually orthogonal. 

  Considering the orbit under $\Aut(\CC)$ of the pair $p(T_{x_0})\subseteq p(T\cM_{x_0})$, \cite{wright2014field} shows that $\Aut(\CC).p(T\cM_{x_0})$ consists of vector spaces that can be summed directly, and in particular, if $p(T\cM_{x_0})^{\tilde \sigma_2}\neq p(T\cM_{x_0})^{\tilde \sigma_1}$ then $p(T\cM_{x_0})^{\tilde \sigma_1} \cap p(T\cM_{x_0})^{\tilde \sigma_2} = \{0\}$.  
  Hence if a Galois automorphism $\tilde \sigma \in \Aut(\CC)$ preserves $p(T_{x_0})$, then 
  $p(T_{x_0}) \subseteq p(T\cM_{x_0})^\sigma\cap p(T\cM_{x_0})$ and hence $p(T\cM_{x_0})^\sigma = p(T\cM_{x_0})$. 
  Consequently, \[\stab_{\Aut(\CC)}(p(T\cM_{x_0})) = \Gal(\CC/K_0)\subseteq \stab_{\Aut(\CC)}(p(T_{x_0})) = \Gal(\CC/K),\] and $K_0\supseteq K$.


\end{proof}
Let $i\ge 1$. 
Since $\stab_{\SL_2(\RR)}(x_i)$ is discrete and Zariski dense, there is a hyperbolic element in $\stab_{\SL_2(\RR)}(x_i)$. Hence we may apply Claim \ref{claim: direct sum} and obtain a number field $K_i \supseteq K$. 
For every embedding $\sigma: K \to \CC$ consider the collection of embeddings extending $\sigma$
\[\Hom_{\sigma}(K_i,\CC) = \{\sigma'\in \Hom(K_i, \CC): \sigma'|_{K} = \sigma\}.\]
It is a standard result in Galois theory that $\#\Hom_{\sigma}(K_i,\CC) = [K_i:K]$ for every $\sigma\in \Hom(K, \CC)$.
Note that 
\begin{align}\label{eq: d' < rk}
  [K_i:K] = \frac{1}{2}\dim \bigoplus_{\sigma'\in \Hom_{{\rm Id}_K}(K_i, \CC)}p(T_{x_i})^{\sigma'} \le \dim p(T\cM) / 2 = \rk \cM.
\end{align} 
Restrict $i$ to a subsequence such that $[K_i:K] = [K_i:\QQ]/d$ is a constant $d' \le \rk \cM$. 

\begin{definition}[Definition of sections of a quasi-projective bundle]
  Let $\sigma\in \Hom(K, \CC)$, and define the locally flat vector bundle $V^\sigma$ by $V^\sigma_{x} = T\cM_{x}^\sigma$ if $\sigma \neq {\rm Id}_K$ and $V^{{\rm Id}_K} = T\cM_{x}\cap T_{x}^\perp$.
  Note that for every $\sigma' \in \Hom_{\sigma}(K_i, \CC)\setminus \{{\rm Id}_{K_i}\}$ we have $T_{x_i}^{\sigma'}\subseteq V_\sigma$. 

  Thus, consider the bundle $P^{\sigma, \circ}$ on $\cM$ defined by 
  \[P^{\sigma, \circ}_x = \left\{V\in {\rm Gr}(V^\sigma_{x},2) : V\cap \ker P = \{0\}\right\}.\]
  For every $i\ge 1$ and an embedding $\sigma' \in \Hom_\sigma(K_i,\CC)\setminus \{{\rm Id}_{K_i}\}$
  define a flat section 
  $\varphi_{i,\sigma'}: \SL_2(\RR).x_i \to P^{\sigma, \circ}$ by 
  \[\varphi_{i,\sigma'}(g.x_i) = g^* T_{x_i}^{\sigma'} = T_{g.x_i}^{\sigma'}. \]
\end{definition}
The vector bundles $V^\sigma$ are studied in \cite{eskin2018algebraichull}, and their algebraic hull is computed:
\begin{align}
  \label{eq: alg hull V sig}
  \begin{split}
  A_{V^\sigma,x} &= \left\{g\in \SL(V^\sigma): g|_{V^\sigma_x\cap \ker p} = {\rm Id}_{V^\sigma_x\cap \ker p}, \tilde Q\circ g = \tilde Q\right\}\\&\cong \begin{pmatrix}
    \Sp(2(\rk(\cM) - \ind_{\sigma = {\rm Id}_K}))&0\\
    *&I_{ V^\sigma\cap \ker p}
  \end{pmatrix}.
  \end{split}
\end{align}
This bundle $P^{\sigma, \circ}$ is not projective, only quasi-projective, so this is not a proof of Part \ref{part: lifts to projective1}. The first way to make it projective, replacing it by the entire Grassmannian, is essentially what we did in Subsection \ref{ssec: work in H11}. It fails if $\dim (V\cap \ker P)\ge 2$. Then the action of $A_{V^\sigma,x}$ has invariant sections and they would disprove Part \ref{part: no section1}. 
However, we can simply take a quotient of $V^\sigma$ to reduce $V^\sigma\cap \ker P$ to be one dimensional.
\begin{claim}\label{claim: group has no inv}
  If $d\ge 2$, the action $\Sp(2d)\acts {\rm Gr}(2d, 2)$ has no invariant measures. 
  If $d\ge 1$, the action $\begin{pNiceArray}{c|c}
    \Sp(2d)&0\\
    \hline *\cdots*&1
  \end{pNiceArray}\acts {\rm Gr}(2d+1, 2)$ has no invariant measures. 
\end{claim}
This is a simple application of \cite[Cor 1.10.]{bader2017almost}.
\begin{definition}\label{def: large size assumption}
  We say that $(\cM, d')$ satisfies the \emph{large size assumption} if 
  For some $\sigma \in \Hom(K, \CC)$ we have $\dim(V^\sigma) \ge 3$ and $\Hom_\sigma(K_i, \CC)\setminus \{{\rm Id}_{K_i}\} \neq \emptyset$. 
\end{definition}
\begin{remark}
  Note that while we used $K_i$ to define the large size assumption, it is only dependant on $d'$ as 
  \[\#(\Hom_\sigma(K_i, \CC)\setminus \{{\rm Id}_{K_i}\}) = \begin{cases}
    d' - 1&\text{if }\sigma = {\rm Id}_{K}\\
    d'&\text{if }\sigma = {\rm Id}_{K}.
  \end{cases}
  \]
  Note that if $K\neq \QQ$ then the large size assumption always holds, with $\sigma \in \Hom(K, \CC)\setminus \{{\rm Id}_K\}$, and $\dim V^\sigma = \dim T\cM > 2$. (If we had $\dim T\cM = 2$ then it would be a closed orbit, and hence periodic. However we assumed that it is not periodic.)
  If $K = \QQ$ then the large size assumption is simply $d' \ge 2$, and $\dim V_{{\rm Id}_K} = \dim (T\cM) - 2 \ge 3$.
\end{remark}
We now study the complement of the large size assumption:
\begin{definition}
  We say that $(\cM, d')$ satisfies the \emph{small size assumption} if $K = \QQ$ and one of the following holds
  \begin{enumerate}
    \item $d' = 1$.
    \item $d' = 2$, $\dim T\cM = 4$, and $\ker p \cap T\cM = \{0\}$.
  \end{enumerate}
\end{definition}
\begin{claim}
  The small and large size assumption are complementary
\end{claim}
\begin{proof}
  If $(\cM, d')$ does not satisfy the large size assumption then we must have $K = \QQ$ and one of the following:
  \begin{enumerate}
    \item $d'=1$,
    \item $d' \ge 2$ and $\dim T\cM = \dim V^{{\rm Id}_\QQ} + 2 \le 4$. 
  \end{enumerate}
  The first case is covered in the small size assumption. For the second, note that
  \begin{align}
    4\le 2d' \stackrel{\eqref{eq: d' < rk}}{\le} 2\rk \cM \le \dim T\cM \le 4
  \end{align}
  Hence all the inequalities are equalities, and we get that $d' = 2$ and $T\cM = p(T\cM)$ and hence $\ker p \cap T\cM = \{0\}$.
\end{proof}
For the rest of the subsection we assume that $(\cM, d')$ satisfies the large size assumption, let  $\sigma \in \Hom(K, \CC)$ as in Definition \label{def: large size assumption} and $\sigma_i \in \Hom_\sigma(K_i, \CC)\setminus \{{\rm Id}_{K_i}\}$. 
Let $U\subseteq V^\sigma\cap \ker p$ be defined as follows: if $V^\sigma\cap \ker p = 0$ then $U=0$. 
If $\ker p\cap V^\sigma$ has positive dimension then since $\ker p$ is the constant vector bundle, there is codimension $1$ subbundle $U < V^\sigma\cap \ker p$. 
\begin{claim}
  $\dim V^\sigma / U \ge 3$.
\end{claim}
\begin{proof}
  If $U=0$ then the result is given by the large size assumption.
  If $U\neq 0$ then \[\dim V^\sigma / U = \dim V^\sigma / (V^\sigma\cap \ker p) + 1 = 
  \dim p(V^\sigma) + 1 \ge \dim p(T^{\sigma_i}) + 1 = 3.
  \]
\end{proof}
By Eq. \eqref{eq: alg hull V sig}, tha algebraic hull $A_{ V^\sigma / U}$ is of the form discussed in Observation \ref{claim: group has no inv}. 
Let $P^{\sigma}$ be the locally constant bundle 
\[P^\sigma_x = {\rm Gr}(V^\sigma_x / U),\]
and the projection $\pi_U:P^{\sigma, \circ}\to P^{\sigma}$ induced by the quotient mod $U$. 
The compositions $\pi_U\circ \varphi_{i,\sigma_i}:\SL_2(\RR).x_i\to P^\sigma$  complete the proof of Part \ref{part: lifts to projective1}. 

Now we cah prove Part \ref{part: no section1}: if there was a measure $\nu$ on $P^\sigma$ that projects to $\lambda_\cM$ on $\cM$, then by Claim \ref{claim: no inv measure} applied to Grassmannians instead of projective spaces, the disintegration $\nu_x$ on $P^\sigma_x$ are $A_{ V^\sigma / U}$ invariant. However, Observation \ref{claim: group has no inv} shows that there is not such invariant measure. This completes the proof of $\ref{part: no section1}$, and hence we get a contradiction to the existence of $(x_i)_{i=1}^\infty$.

\subsection{Orbits in affine manifolds with the small size assumption}
\label{ssec: arithmeticity}
For this subsection we return to work with absolute and relative cohomology with coefficients in $\RR$. 
Assume that $(\cM, d')$ satisfies the small size assumption. 
In particular, $K = \QQ$. We show the following theorem that contradicts the existence of the sequence $(x_i)_{i=1}^\infty$. 
\begin{theorem}\label{thm: stab is a lattice}
  For every orbit $\SL_2(\RR).{x_0}$ in $\cM$ such that $\stab_{\SL_2(\RR)}(x_0)$ contains a hyperbolic element $g_0$ and $T_{x_0}$ is defined over a number field $K_0$ of degree $d'$ over $\QQ$, we have that $\stab_{\SL_2(\RR)}(x_0)$ is a lattice. 
\end{theorem}
The result of the theorem contradicts the assumption that $\stab_{\SL_2(\RR)}(x_i)$ is not a lattice. 
The proof uses ideas of McMullen \cite{mcmullen2003teichmuller,mcmullen2003genus2dynamics}. 
The goal is to use to use $g_0$ and the definition over $K_0$ to identify a property of $x_0$ that holds only on a closed subset of $\cM$ and has dimension $3$, thereby proving that the orbit is closed. In the $d'=1$ case, the closed set is the collection of flat surfaces that cover a torus. 
In the $d'=2$ case, the elements of the closed set are usually termed \emph{eigenforms}.
\begin{proof}[Proof of Theorem \ref{thm: stab is a lattice} when $d' = 1$]
  Since $K = \QQ$, we get that $T\cM_{x_0}\cap H^1_\rel({x_0};\ZZ)$ is of full rank in $T\cM_{x_0}$. 
  Since $d' = 1$ and we get that $K_0 = \QQ$. 
  Hence $T_{x_0}\subseteq T\cM_{x_0}$ is defined over $\ZZ$.
  Let $\Lambda_{x_0} = \{H^1_\rel(x_0;\ZZ) \cap T_{x_0}\}$ be a full rank $2$ lattice in $T_{x_0}$. 
  Let $D = \cov_{\tilde Q}(\Lambda_{x_0}) = |{\tilde Q}(v_1, v_2)| \in \NN_{>0}$, where $v_1, v_2$ are a basis of $\Lambda$. 
  Let 
  \[Z_D = \left\{x \in \cM: \begin{matrix}
    \Lambda_{x} = H^1_\rel(x;\ZZ) \cap T_{x}\text{ has full rank $2$ in $T_{x}$}\\
    \cov_{\tilde Q}(\Lambda_{x}) = D
  \end{matrix}\right\}.\]
  Note that $Z_D$ is $\SL_2(\RR)$ invariant. 
  To show that it is a $3$ dimensional closed manifold, suppose that $z_i \in Z_D$ converges to $z_\infty \in \cM$. We will show that for all $i>0$ sufficiently large, $z_i = g_i.z_\infty$ for some $g_i\in \SL_2(\RR)$ with $g_i\to I_2$. 
  Indeed, since $H^1_\rel(-,\CC)$ is locally flat, identify 
  $H^1_\rel(z_i,\CC)$ with $H^1_\rel(z_\infty,\CC)$ for all $i$ sufficiently large. 
  Then $T_{z_i}\subseteq T\cM_{z_\infty}\subseteq H^1_\rel(z_\infty,\CC)$. 
  For every $i=1,\dots, \infty$ there are two prescribed elements $(dx)_i = {\rm Re}(\omega_i), (dy)_i = {\rm Im}(\omega_i) \in T_{z_i}$, where $\omega_i$ is the holomorphic form corresponding to $z_i$, with $\tilde Q((dx)_i, (dy)_i) = 1$. 
  The limit $z_i\to z_\infty$ implies that 
  \begin{align}\label{eq: limit forms}
    (dx)_i\to (dx)_\infty\text{ and }(dy)_i\to (dy)_\infty. 
  \end{align}
  We will use the following Claim:
  \begin{claim}\label{claim: discreteness of D-integral planes}
    For every integral $2$-form $\omega$ on $\ZZ^d$, let 
    \[U_\omega = \{V\in {\rm Gr}(\RR^d, 2): \omega|_V\text{ is non-degenerate}\},\] 
    and for every $D\ge 1$
    \[U_\omega(\ZZ,D) = \left\{V\in U_\omega: \begin{matrix}
      V\cap \ZZ^d\text{ has full rank $2$ in }V,\\
      \text{and }\cov_\omega(V\cap \ZZ^d) = D
    \end{matrix}\right\}.\] 
    Then $U_\omega(\ZZ,D)$ is discrete in $U_\omega$. 
  \end{claim}
  Using the claim for a general lattice $H^1_\rel(z_\infty,\ZZ) \cap T\cM_{z_\infty}$ instead of $\ZZ^d$ we note that $T_{z_i} \in U_{\tilde Q}(\ZZ, D)$ and $T_{z_\infty} \in U_{\tilde Q}$. 
  The discreteness of $U_{\tilde Q}(\ZZ, D)$ implies that $T_{z_i} = T_{z_\infty}$ for all $i$ sufficiently large. We obtain that for some $a_i,b_i,c_i,d_i \in \RR$ we have 
  \begin{align}
    (dx)_i = a_i (dx)_\infty + b_i (dy)_\infty,\qquad
    (dy)_i = c_i (dx)_\infty + d_i (dy)_\infty.
  \end{align}
  By Eq. \eqref{eq: limit forms}, $a_i, d_i \to 1$, $b_i, c_i \to 0$. 
  Since $\tilde Q((dx)_i, (dy)_i) = 1$ we obtain that $g_i = \begin{pmatrix}
    a_i&b_i\\
    c_i&d_i\\
  \end{pmatrix}$ has determinant $1$. 
  By the definition of the $\SL_2(\RR)$ action we obtain $z_i = g_i.z_\infty$. 
  This implies that $Z_D$ is closed and locally an $\SL_2(\RR)$ orbit. 
  Hence $\SL_2(\RR).x_0\subseteq Z_D$ is a closed orbit. By Smillie's theorem \cite{veech1995geometric}, $\stab_{\SL_2(\RR)}(x_0)$ is a lattice.
\end{proof}
\begin{proof}[Proof of Claim \ref{claim: discreteness of D-integral planes}]
  Let $V\in U_\omega(\ZZ, D)$, and choose a basis $v_1, v_2\in V\cap \ZZ^d$ for which $\omega(v_1, v_2) = D$. 
  observe the operator $A_V:\RR^d\to \RR^d$ defined by $A_V(v) = \omega(v, v_2) - \omega(v, v_1)$ restricts to multiplication by $D$ on $V$ and $0$ on $V^\perp$. 

  For $V\in U_\omega$, the map $V\mapsto V^\perp$ is continuous, and $V\oplus V^\perp = \RR^d$. 
  Hence there is a unique map $A_V$ such that $A_V|_{V} = D\cdot {\rm Id}_V$ and $A_V|_{V^\perp} = 0$. 
  In addition, $A_V$ depends continuously on $V$. 
  However, for $V$ in $U_\omega(\ZZ, D)$, the operator $A_V$ is integral. 
  
  Assume that $V_i \in U_\omega(\ZZ, D)$ converge to $V\in U_\omega$. 
  Then $A_{V_i} \to A_V$. Since $A_{V_i}$ is integral, we obtain that $A_{V_i}$ is constant for all $i$ sufficiently large. Since $V_i$ is the $D$-eigenspace of $A_{V_i}$, we deduce that $V_i$ is constant for $i$ sufficiently large. 
  This shows that $U_\omega(\ZZ, D)$ is discrete in $U_\omega$. 
\end{proof}

\begin{proof}[Proof of Theorem \ref{thm: stab is a lattice} when $d' = 2$]
  In this case, $\dim T\cM = 4$ and $\ker p \cap T\cM = \{0\}$.
  In particular, $\tilde Q|_{T\cM}$ is nondegenerate. 
  Since $T_{x_0}$ is defined over $K_0$ and it is real, we deduce that $K_0$ is a totally real quadratic field. 
  Let $\sigma_0$ be the algebraic conjugation on $K_0$. 
  By Claim \ref{claim: direct sum} we deduce that $T_{x_0}^{\sigma_0} = T_{x_0}^\perp \cap T\cM$.
  As in the proof of Claim \ref{claim: direct sum} if $g_0$ is a hyperbolic element in $\stab_{\SL_2(\RR)}(x_0)$, with largest eigenvalue $\lambda_0>0$
  the integral operator $A = g_0^* + (g_0^*)^{-1}|_{T\cM_{x_0}}$ on $T\cM_{x_0}$ that satisfies that $A|_{T_{x_0}} \equiv \lambda_0+\lambda_0^{-1}$ and $A|_{T_{x_0}^{\sigma_0}} \equiv \sigma(\lambda_0+\lambda_0^{-1})$. 
  Denote $\eta = \lambda_0+\lambda_0^{-1}$ and note that since it is an eigenvalue of an integral operator $A$, it is an algebraic integer in $K$. 
  The proof of this part of Theorem \ref{thm: stab is a lattice} ends similarly to the $d'=1$ case provided the following claim:
  \begin{claim}\label{claim: discreteness of eta-integral planes}
    For every nondegenerate integral $2$-form $\omega$ on $\ZZ^4$, let 
    \[U_\omega = \{V\in {\rm Gr}(\RR^4, 2): \omega|_V\text{ is non-degenerate}\},\] 
    and
    \[U_\omega(\ZZ, \eta, \sigma(\eta)) = \left\{V\in U_\omega: \exists A\in M_{2\times 2}(\ZZ), A|_{V} = \eta {\rm Id}_V, A|_{V^\perp} = \sigma(\eta) {\rm Id}_{V^\perp}\right\}.\] 
    Then $U_\omega(\ZZ, \eta, \sigma(\eta))$ is discrete in $U_\omega$. 
  \end{claim}  
  The Proof of this claim is similar to the proof of Claim \ref{claim: discreteness of D-integral planes}, and we will not repeat it. 
\end{proof}

\appendix
\section{Proof Proposition \ref{prop: impossible section method1}}
\label{appendix: proof impossible section strat}
Let $M$ and $\lambda$ as in Proposition \ref{prop: impossible section method1}, and $(\SL_2(\RR).x_i)_{i=1}^\infty$ is an infinite sequence of orbits, with $\delta(\stab_{\SL_2(\RR)}(x_i))\to 1$. 
By Proposition \cite[Prop. 4.1]{solan2024critical}, for every $i$ there there is a $\ta$-invariant probability measure $\mu_i$ on $\SL_2(\RR).x_i$ such that $h_{\mu_i}(\ta(1))\ge 1 - 2(1-\delta(\stab_{\SL_2(\RR)}(x_i)))\to 1$. 
By Part \ref{part: equidistribution1}, we have that $\mu_i \to \lambda$. 
By \cite[Thm~7.6 (ii)]{einsiedler2010diagonal} the entropy coincides with the leafwise dimensions,
\begin{align}
  h_{\tilde \mu_i}(\ta(1)) = \dim^\tu(\mu_i) = \dim^{\tu^t}(\mu_i) \to 1.
\end{align}
Let $\pi:P\to M$ be the projective bundle provided by Part \ref{part: lifts to projective1}, and $\varphi_i:\SL_2(\RR).x_i\to P$ the invariant sections.
Since leafwise dimension is preserved by embeddings, we deduce that 
\begin{align}
  \dim^\tu((\varphi_i)_*\mu_i) = \dim^{\tu^t}((\varphi_i)_*\mu_i) \to 1.
\end{align}
Take a weak-$*$ partial limit $\nu_\infty$ of $((\varphi_i)_*\mu_i)_{i=1}^\infty$. 
By Solan \cite[Thm 3.9]{solan2024critical}, we deduce that $\nu_\infty$ is $\tu$ invariant and $\tu^t$-invariant. Hence it is $\SL_2(\RR)$-invariant. 
Since the projection $\pi:P\to X$ is proper, the pushforward map $\pi_*$ is continuous with respect to the weak-$*$ topology. 
Hence $\pi_*\nu_\infty = \lim_{i\to \infty} \mu_i = \lambda$ and $\nu_\infty$ is a probability measure. 
This contradicts Part \ref{part: no section1}, and derive the desired contradiction.



\bibliographystyle{plain}
\bibliography{BibErg}{}

\begin{thebibliography}{10}

\bibitem{avila2017symplectic}
Artur Avila, Alex Eskin, and Martin M{\"o}ller.
\newblock Symplectic and isometric {{{$\SL_2(\RR)$}}}-invariant subbundles of the {{H}}odge bundle.
\newblock {\em Journal f{\"u}r die reine und angewandte Mathematik (Crelles Journal)}, 2017(732):1--20, 2017.

\bibitem{bader2017almost}
Uri Bader, Bruno Duchesne, and Jean L{\'e}cureux.
\newblock Almost algebraic actions of algebraic groups and applications to algebraic representations.
\newblock {\em Groups Geom. Dyn}, 11(2):705--738, 2017.

\bibitem{bader2021arithmeticity}
Uri Bader, David Fisher, Nicholas Miller, and Matthew Stover.
\newblock Arithmeticity, superrigidity, and totally geodesic submanifolds.
\newblock {\em Annals of mathematics}, 193(3):837--861, 2021.

\bibitem{calta2004veech}
Kariane Calta.
\newblock Veech surfaces and complete periodicity in genus two.
\newblock {\em Journal of the American Mathematical Society}, 17(4):871--908, 2004.

\bibitem{einsiedler2010diagonal}
M~Einsiedler and E~Lindenstrauss.
\newblock Diagonal actions on locally homogeneous spaces.
\newblock {\em Homogeneous flows, moduli spaces and arithmetic}, 10:155--241, 2010.

\bibitem{eskin2018algebraichull}
Alex Eskin, Simion Filip, and Alex Wright.
\newblock The algebraic hull of the {K}ontsevich--{Z}orich cocycle.
\newblock {\em Annals of Mathematics}, 188(1):281--313, 2018.

\bibitem{eskin2018invariant}
Alex Eskin and Maryam Mirzakhani.
\newblock Invariant and stationary measures for the action on moduli space.
\newblock {\em Publications math{\'e}matiques de l'IH{\'E}S}, 127:95--324, 2018.

\bibitem{eskin2015isolation}
Alex Eskin, Maryam Mirzakhani, and Amir Mohammadi.
\newblock Isolation, equidistribution, and orbit closures for the {{{$\SL (2, \RR)$}}} action on moduli space.
\newblock {\em Annals of Mathematics}, pages 673--721, 2015.

\bibitem{forni2002deviation}
Giovanni Forni.
\newblock Deviation of ergodic averages for area-preserving flows on surfaces of higher genus.
\newblock {\em Annals of Mathematics}, 155(1):1--103, 2002.

\bibitem{lindenstrauss2022effective}
Elon Lindenstrauss, Amir Mohammadi, and Zhiren Wang.
\newblock Effective equidistribution for some one parameter unipotent flows.
\newblock {\em arXiv preprint arXiv:2211.11099}, 2022.

\bibitem{mcmullen2003billiards}
Curtis McMullen.
\newblock Billiards and teichm{\"u}ller curves on hilbert modular surfaces.
\newblock {\em Journal of the American Mathematical Society}, 16(4):857--885, 2003.

\bibitem{mcmullen2003teichmuller}
Curtis~T McMullen.
\newblock Teichm{\"u}ller geodesics of infinite complexity.
\newblock 2003.

\bibitem{mcmullen2003genus2dynamics}
Curtis~T. McMullen.
\newblock Dynamics of {{$\SL_2(\RR)$}} over moduli space in genus two.
\newblock {\em Annals of Mathematics}, 165(2):397--456, 2007.

\bibitem{mohammadi2023isolations}
Amir Mohammadi and Hee Oh.
\newblock Isolations of geodesic planes in the frame bundle of a hyperbolic 3-manifold.
\newblock {\em Compositio Mathematica}, 159(3):488--529, 2023.

\bibitem{mozes1995space}
Shahar Mozes and Nimish Shah.
\newblock On the space of ergodic invariant measures of unipotent flows.
\newblock {\em Ergodic theory and dynamical systems}, 15(1):149--159, 1995.

\bibitem{rached2024separation}
John Rached.
\newblock Separation of horocycle orbits on moduli space in genus 2.
\newblock {\em arXiv preprint arXiv:2406.19527}, 2024.

\bibitem{sanchez2023effective}
Anthony Sanchez.
\newblock Effective equidistribution of large dimensional measures on affine invariant submanifolds.
\newblock {\em arXiv preprint arXiv:2306.06740}, 2023.

\bibitem{solan2024critical}
Omri~Nisan Solan.
\newblock Critical exponent gap and leafwise dimension.
\newblock {\em arXiv preprint arXiv:2404.00700}, 2024.

\bibitem{veech1995geometric}
William~A Veech.
\newblock Geometric realizations of hyperelliptic curves.
\newblock In {\em Algorithms, fractals, and dynamics}, pages 217--226. Springer, 1995.

\bibitem{wright2014field}
Alex Wright.
\newblock The field of definition of affine invariant submanifolds of the moduli space of abelian differentials.
\newblock {\em Geometry \& Topology}, 18(3):1323--1341, 2014.

\end{thebibliography}
\end{document}